\documentclass[11pt,twoside, final]{amsart}
\usepackage{amssymb}
\usepackage{amsmath,amsthm,amscd,amsfonts,amssymb,enumerate}
\usepackage{graphicx}		
\usepackage{color}
\usepackage[colorlinks]{hyperref}

\newtheorem{theorem}{Theorem}[section]

\theoremstyle{definition}
\newtheorem{definition}[theorem]{Definition}

\theoremstyle{remark}
\newtheorem{remark}[theorem]{Remark}
\numberwithin{equation}{section}
\begin{document}
\title[Extensions of beta and related functions...]{Certain new extensions of beta and related functions}

\author[Mohd Ghayasuddin]{Mohd Ghayasuddin}
\address[Mohd Ghayasuddin]{Department of Mathematics, Integral University Campus, Shahjahanpur-242001, India}
\email{ghayas.maths@gmail.com}

\author[Musharraf Ali]{Musharraf Ali}
\address[Musharraf Ali]{Department of Mathematics, G.F. College, Shahjahanpur-242001, India}
\email{drmusharrafali@gmail.com}

\author[R. B. Paris]{R. B. Paris $^*$}
\address[R. B. Paris]{Division of Computing and Mathematics, Abertay University, Dundee DD1 1HG, UK}
\email{r.paris@abertay.ac.uk}

%% If there are three of more authors they are added in the obvious way.

  \thanks{$^*$Corresponding author}

\vspace{.50cm}
\parindent=5mm
 \maketitle
%
%%% The following environment is needed for the abstract.
%%%

\begin{abstract}
In this paper, we introduce and investigate a new extension of the beta function by means of an integral operator involving a product of Bessel-Struve kernel functions. We also define a new extension of the well-known beta distribution, the Gauss hypergeometric function and the confluent hypergeometric function in terms of our extended beta function. In addition, some useful properties of these extended functions are also indicated in a systematic way.\\

\end{abstract}

\vspace{.50cm}
\parindent=8mm
\section{\bf Introduction}
Throughout in this paper, let $\mathbb{N}$, $\mathbb{R}$ and $\mathbb{C}$ be the sets of natural numbers, real numbers and complex numbers, respectively, and let\\
$$\mathbb{N}:=\{1, 2, 3,...\},~{\mathbb{N}}_{0}:=\{0, 1, 2, 3,...\}=\mathbb{N}\cup \{0\}.$$

The classical beta function $B(\xi_{1},\xi_{2})$ is defined by (see \cite{8}, see also \cite{10})
\begin{equation} \label{1.1}
B(\xi_{1},\xi_{2})=\int_{0}^{1} y^{\xi_{1}-1}~(1-y)^{\xi_{2}-1} dy
\end{equation}
$$(\Re(\xi_{1})>0,~\Re(\xi_{2})>0).$$
In 1997, Chaudhry {\it et al.} \cite{1} introduced a very useful generalization of the classical beta function \eqref{1.1} by
\begin{equation} \label{1.2}
B_{p}(\xi_{1},\xi_{2})=\int_{0}^{1} y^{\xi_{1}-1}~(1-y)^{\xi_{2}-1}~\exp\left[-\frac{p}{y(1-y)}\right]dy
\end{equation}
$$(\Re(\xi_{1})>0,~\Re(\xi_{2})>0,~\Re(p)>0).$$
Obviously, for $p=0$, \eqref{1.2} reduces to \eqref{1.1}. The most interesting applications of \eqref{1.2} are given by Chaudhry {\it et al.} in \cite{2}. They generalized the classical Gauss and confluent hypergeometric functions by means of the extended beta function $B_{p}(\xi_{1},\xi_{2})$ as follows:
\begin{equation} \label{1.3}
 \aligned & F_{p}(\xi_{1}, \xi_{2}; \xi_{3}; x)=\sum_{n=0}^{\infty} \frac{(\xi_{1})_{n}~B_{p}(\xi_{2}+n, \xi_{3}-\xi_{2})}{B(\xi_{2}, \xi_{3}-\xi_{2})} \frac{x^{n}}{n!}\\
& \hskip 7mm (p\geq0,~|x|<1,~\Re(\xi_{3})>\Re(\xi_{2})>0)
 \endaligned
\end{equation}
and
\begin{equation} \label{1.4}
\aligned & \Phi_{p}(\xi_{2}; \xi_{3}; x)=\sum_{n=0}^{\infty} \frac{B_{p}(\xi_{2}+n, \xi_{3}-\xi_{2})}{B(\xi_{2},
\xi_{3}-\xi_{2})}~\frac{x^{n}}{n!}\\
& \hskip 7mm  (p\geq0,~\Re(\xi_{3})>\Re(\xi_{2})>0).
\endaligned
\end{equation}

Among the many interesting properties given in \cite{2}, the following integral representations are recalled:
\begin{equation} \label{1.5}
 F_{p}(\xi_{1}, \xi_{2}; \xi_{3}; x)
 =\frac{1}{B(\xi_{2}, \xi_{3}-\xi_{2})}
\end{equation}
\begin{equation*}\times \int_{0}^{1}
y^{\xi_{2}-1}~(1-y)^{\xi_{3}-\xi_{2}-1}~(1-xy)^{-\xi_{1}}  \exp\left[-\frac{p}{y(1-y)}\right]dy
\end{equation*}
$$(p\geq0,~|\arg(1-x)|<\pi,~\Re(\xi_{3})>\Re(\xi_{2})>0)$$
and
\begin{equation} \label{1.6}
\Phi_{p}(\xi_{2}; \xi_{3}; x)=\frac{1}{B(\xi_{2}, \xi_{3}-\xi_{2})}
\end{equation}
\begin{equation*} \times \int_{0}^{1}
y^{\xi_{2}-1}~(1-y)^{\xi_{3}-\xi_{2}-1}~\exp\left[xy-\frac{p}{y(1-y)}\right]dy
\end{equation*}
$$(p\geq0, ~\Re(\xi_{3})>\Re(\xi_{2})>0).$$
If we set $p=0$ in \eqref{1.5} and \eqref{1.6} then we easily recover the integral representations of the classical Gauss and confluent hypergeometric functions as follows (see \cite{8} and also \cite{10}):
\begin{equation} \label{1.5r}
 F(\xi_{1}, \xi_{2}; \xi_{3}; x)
 =\frac{1}{B(\xi_{2}, \xi_{3}-\xi_{2})} \int_{0}^{1}
y^{\xi_{2}-1}~(1-y)^{\xi_{3}-\xi_{2}-1}~(1-xy)^{-\xi_{1}}dy
\end{equation}
$$(|\arg(1-x)|<\pi,~\Re(\xi_{3})>\Re(\xi_{2})>0)$$
and
\begin{equation} \label{1.6r}
\Phi(\xi_{2}; \xi_{3}; x)=\frac{1}{B(\xi_{2}, \xi_{3}-\xi_{2})} \int_{0}^{1}
y^{\xi_{2}-1}~(1-y)^{\xi_{3}-\xi_{2}-1}~\exp(xy)dy
\end{equation}
$$(\Re(\xi_{3})>\Re(\xi_{2})>0).$$

By introducing an additional parameter $q$, Choi {\it et al.} \cite{3} defined a further extension of \eqref{1.2} as follows:
\begin{equation} \label{1.2a}
B_{p,q}(\xi_{1},\xi_{2})=\int_{0}^{1} y^{\xi_{1}-1}~(1-y)^{\xi_{2}-1}~\exp\left[-\frac{p}{y}-\frac{q}{(1-y)}\right]dy
\end{equation}
$$(\Re(\xi_{1})>0,~\Re(\xi_{2})>0,~\Re(p)>0,~\Re(q)>0).$$
The case $q=p$ in \eqref{1.2a}, yields the extended beta function given in \eqref{1.2}.\\

Since the beta function and its extensions play a crucial role in the study of special functions, a number of researchers have introduced and investigated several extensions of this important function (see, for example, \cite{m1}--\cite{1}, \cite{m3}, \cite{6}, \cite{7}, \cite{9}).\\

The Bessel-Struve kernel function $S_{\eta} (\lambda t), \lambda \in \mathbb{C}$ is the unique solution of the initial value problem
$L_{\eta}u(t) = \lambda^{2}u(t)$ subject to the initial conditions $ u(0) = 1$ and
$ u'(0) = \frac{\lambda \Gamma({\eta} + 1)}{\sqrt{\pi}~\Gamma({\eta} + \frac{3}{2})}$, where 
\[L_\eta=\frac{d^2u(t)}{dt^2}+\frac{2\eta+1}{t}\left(\frac{du(t)}{dt}-\frac{du(0)}{dt}\right)\]
is the Bessel-Struve differential operator. This function is given by (see \cite{4} and also \cite{5})
\begin{equation*}
S_{\eta}(\lambda t) = j_{\eta}(i\lambda t) - i h_{\eta}(i\lambda t),\qquad \forall~ t \in \mathbb{C},
\end{equation*}
where $j_{\eta}$ and $ h_{\eta} $ are the normalized Bessel and Struve functions. The series representation of the Bessel-Struve kernel
function is given as follows:
\begin{equation}\label{1.7}
S_{\eta} (t) =\frac{\Gamma({\eta} + 1)}{\sqrt{\pi}} \sum_{m=0}^{\infty} \frac{ t^{m} {\Gamma(\frac{m+1}{2})}}{~m!~\Gamma(\frac{m}{2} + \eta +1)}.
\end{equation}
Also, we have the following relations of the Bessel-Struve kernel function with the exponential function (see \cite{4} and also \cite{5}):
\begin{equation}\label{1.8}
S_{-\frac{1}{2}}(t) = e^{t}~{\rm and}~S_{\frac{1}{2}}(t) = \frac{e^{t} - 1}{t}.
\end{equation}

The main object of this paper is to introduce and investigate a new extension of the beta function by making use of the Bessel-Struve kernel function \eqref{1.7}. This is applied to extend the well-known beta distribution arising in statistical distribution theory. We also define a new class of Gauss and confluent hypergeometric functions in terms of our introduced beta function.

%%%
%%% Section 2
%%%
\section{\bf Extended beta function and its properties}
\parindent=8mm This section deals with a new extension of the beta function and its associated properties.

\begin{definition}
The new extended beta function $B_{\eta}^{p,q}(\xi_{1},\xi_{2})$ for $\Re(\eta)>-1$ is defined by
\begin{equation} \label{2.1}
B_{\eta}^{p,q}(\xi_{1},\xi_{2})=\int_{0}^{1}y^{\xi_{1}-1} (1-y)^{\xi_{2}-1} S_{\eta}\left[-\frac{p}{y}\right] S_{\eta}\left[-\frac{q}{1-y}\right] dy
\end{equation}
$$(\Re(\xi)_{1}>0,~\Re(\xi_{2})>0,~\Re(p)>0,~\Re(q)>0,~\Re(\eta)>-1)$$
where $S_{\eta}(t)$ denotes the Bessel-Struve kernel function given by \eqref{1.7}.
\end{definition}

\begin{remark}
We note that the case $\eta=-\frac{1}{2}$ in \eqref{2.1} yields the extended beta function defined by Choi {\it et al.} \cite{3}, which further for $q=p$ gives the known extension of the beta function given by Chaudhry {\it et al.} \cite{1}. Obviously, when $p=q=0$, \eqref{2.1} reduces to the classical beta function \eqref{1.1}.
\end{remark}

\begin{theorem}
The following integral representations for the extended beta function $B_{\eta}^{p,q}(\xi_{1},\xi_{2})$ hold true:
\begin{equation} \label{2.5}
B_{\eta}^{p,q}(\xi_{1},\xi_{2})=2 \int_{0}^{\frac{\pi}{2}} {\cos^{2\xi_{1}-1} t}~{\sin^{2\xi_{2}-1} t}~S_{\eta}(-p \sec^{2}t)~~S_{\eta}(-q \csc^{2}t) dt,
\end{equation}
\begin{equation} \label{2.5a}
B_{\eta}^{p,q}(\xi_{1},\xi_{2})=\int_{0}^{\infty} \frac{w^{\xi_{1}-1}}{(1+w)^{\xi_{1}+\xi_{2}}}~S_{\eta}\left[-\frac{p(1+w)}{w}\right]~S_{\eta}\left[-q(1+w)\right] dw,
\end{equation}
\begin{equation}\label{2.8}B_{\eta}^{p,q}(\xi_{1},\xi_{2})=2^{1-\xi_{1}-\xi_{2}} \int_{-1}^{1} (1+w)^{\xi_{1}-1}~(1-w)^{\xi_{2}-1}
\end{equation}
\begin{equation}\nonumber \times S_{\eta}\left[-\frac{2p}{1+w} \right]~S_{\eta}\left[-\frac{2q}{1-w} \right] dw,
\end{equation}
\begin{equation}\label{2.8a}
B_{\eta}^{p,q}(\xi_{1},\xi_{2})=(c-a)^{1-\xi_{1}-\xi_{2}}~\int_{a}^{c}~(w-a)^{\xi_{1}-1} (c-w)^{\xi_{2}-1}
\end{equation}
\begin{equation}\nonumber
\times S_{\eta}\left[-\frac{p(c-a)}{(w-a)} \right]~S_{\eta}\left[-\frac{q(c-a)}{(c-w}) \right] dw.
\end{equation}

\end{theorem}
\begin{proof}
On setting  $y=\cos^{2}t$, $y=\frac{w}{1+w}$, $y=\frac{1+w}{2}$ and  $y=\frac{w-a}{(c-a)}$ in \eqref{2.1} we obtain, respectively, the above integral representations \eqref{2.5}-\eqref{2.8a}.
\end{proof}

\begin{theorem}
The following relation for the extended beta function $B_{\eta}^{p,q}(\xi_{1},\xi_{2})$ holds true:
\begin{equation}\label{3.1}  B_{\eta}^{p,q}(\xi_{1},\xi_{2})=B_{\eta}^{p,q}(\xi_{1}+1,\xi_{2}) + B_{\eta}^{p,q}(\xi_{1},\xi_{2}+1)
\end{equation}
\begin{equation*}
  (\Re(p)> 0,\Re(q)> 0, \Re(\eta)>-1).
\end{equation*}
\end{theorem}

\begin{proof} From \eqref{2.1}, we have
\begin{equation*}
B_{\eta}^{p,q}(\xi_{1},\xi_{2})=\int_{0}^{1} y^{\xi_{1}-1}~(1-y)^{\xi_{2}-1}\{y+(1-y)\}~S_{\eta}\left[-\frac{p}{y}\right]S_{\eta}\left[-\frac{q}{1-y}\right]dy,
\end{equation*}
whence
\begin{equation*}
 B_{\eta}^{p,q}(\xi_{1},\xi_{2})=  B_{\eta}^{p,q}(\xi_{1}+1,\xi_{2}) + B_{\eta}^{p,q}(\xi_{1},\xi_{2}+1),
 \end{equation*}
which is our desired result.
\end{proof}

\begin{theorem}
  The extended beta function $B_{\eta}^{p,q}(\xi_{1},\xi_{2})$ satisfies the following summation formula:
\begin{equation}\label{3.5}
  B_{\eta}^{p,q}(\xi_{1},1-\xi_{2})=\sum_{l=0}^{\infty}\frac{(\xi_{2})_{l}}{l!}  B_{\eta}^{p,q}(\xi_{1}+l,1)
  \end{equation}
\begin{equation*}
  (\Re(p)> 0,\Re(q)> 0, \Re(\eta)>-1).
\end{equation*}
\end{theorem}

\begin{proof}
We have
  \begin{equation}\label{3.6}  (1-y)^{-\xi_{2}}= \sum_{l=0}^{\infty}\frac{(\xi_{2})_{l}}{l!}~y^{l}\qquad (|y|<1),
\end{equation}
where $(a)_\ell=\Gamma(a+\ell)/\Gamma(a)$ is the Pochhammer symbol.
Therefore \eqref{2.1} can be written as
\begin{equation*}
B_{\eta}^{p,q}(\xi_{1},1-\xi_{2})=\int_{0}^{1} y^{\xi_{1}-1}~~\left[\sum_{l=0}^{\infty}\frac{(\xi_{2})_{l}}{l!} ~~y^{l}\right] S_{\eta}\left[-\frac{p}{y}\right] S_{\eta}\left[-\frac{q}{1-y}\right] dy.
\end{equation*}
Interchanging the order of integration and summation (which is permissible due to the uniform convergence) in the last expression and further by using \eqref{2.1}, we easily obtain the stated result \eqref{3.5}.
\end{proof}

\begin{theorem}
The extended beta function $B_{\eta}^{p,q}(\xi_{1},\xi_{2})$ satisfies the following summation formula:
\begin{equation}\label{3.7} B_{\eta}^{p,q}(\xi_{1},\xi_{2})=\sum_{l=0}^{\infty} B_{\eta}^{p,q}~~(\xi_{1}+l,\xi_{2}+1)
\end{equation}
\begin{equation*}
  (\Re(p)> 0, \Re(q)> 0, \Re(\eta)>-1).
\end{equation*}
\end{theorem}

\begin{proof} By using the fact
\begin{equation*}
  (1-y)^{\xi_{2}-1}=(1-y)^{\xi_{2}}~~\sum\limits_{l=0}^{\infty} y^{l}\qquad (|y|<1),
\end{equation*}
in \eqref{2.1}, we easily obtain the stated result \eqref{3.7}.
\end{proof}

%%%
%%% Section 4
%%%
\section{\bf An extended beta distribution }

In statistical distribution theory, we define an extended beta distribution as follows:

\begin{equation}\label{4.1}
f(y)=\left\{\begin{array}{ccc}
              \frac{1}{B_{\eta}^{p,q}(\xi_{1},\xi_{2})}~y^{\xi_{1}-1}~(1-y)^{\xi_{2}-1}~S_{\eta}\left[-\frac{p}{y}\right]~S_{\eta}\left[-\frac{q}{(1-y)}\right] & (0<y<1) & ~ \\
              ~ & ~ & ~ \\
              0 & {\rm otherwise} & ~
            \end{array}
 \right.
\end{equation}
$$(p,q > 0, -\infty <\xi_{1},\xi_{2} < \infty ~~,\Re(\eta)>-1).$$
We now discuss some fundamental properties of the extended beta distribution \eqref{4.1}.\\

If $n$ is any real number, then the $n$th moment of $X$ is given by
\begin{equation}\label{4.2}
E(X^{n})=\frac{B_{\eta}^{p,q}(\xi_{1}+n,\xi_{2})}{B_{\eta}^{p,q}(\xi_{1},\xi_{2})}
\end{equation}
$$(\xi_{1},~\xi_{2} \in \mathbb{R},~p,q \in \mathbb{R^{+}},~\Re (\eta)>-1).$$
The particular case of \eqref{4.2} for $n=1$ yields the mean of our proposed extended beta distribution, that is
\begin{equation}\label{4.3}
E(X)=\frac{B_{\eta}^{p,q}(\xi_{1}+1,\xi_{2})}{B_{\eta}^{p,q}(\xi_{1},\xi_{2})}.
\end{equation}
The variance of our introduced distribution can be expressed as
$$Var(X)=E(X^{2})-[E(X)]^{2}=E[(X-E(X))^{2}]$$
\begin{equation}\label{4.4}
=\frac{B_{\eta}^{p,q}(\xi_{1}+2,\xi_{2})~B_{\eta}^{p,q}(\xi_{1},\xi_{2})-[B_{\eta}^{p,q}(\xi_{1}+1,\xi_{2})]^{2}}{[B_{\eta}^{p,q}(\xi_{1},\xi_{2})]^{2}}.
\end{equation}

The coefficient of variation of this distribution (which is defined as the ratio of the standard deviation and mean) can be expressed as
\begin{equation}\label{4.5}
C.V=\sqrt{\frac{B_{\eta}^{p,q}(\xi_{1}+2,\xi_{2})~B_{\eta}^{p,q}(\xi_{1},\xi_{2})}{[B_{\eta}^{p,q}(\xi_{1}+1,\xi_{2})]}-1}.
\end{equation}

The moment generating function (m.g.f.) about the origin of this distribution is given by
\begin{equation*}
M_{X}(t)=\sum_{n=0}^{\infty} \frac{t^{n}}{n!} E(X^{n}),
\end{equation*}
whence
\begin{equation}\label{4.6}
M_{X}(t)=\frac{1}{B_{\eta}^{p,q}(\xi_{1},\xi_{2})}\sum_{n=0}^{\infty} B_{\eta}^{p,q}(\xi_{1}+n,\xi_{2})\frac{t^{n}}{n!}.
\end{equation}

The characteristic function of the proposed distribution can be calculated as follows:
\begin{equation*}
E(e^{itx})=\sum_{n=0}^{\infty} \frac{i^{n}t^{n}}{n!} E(X^{n})
\end{equation*}
\begin{equation}\label{4.7}
E(e^{itx})=\frac{1}{B_{\eta}^{p,q}(\xi_{1},\xi_{2})}\sum_{n=0}^{\infty} B_{\eta}^{p,q}(\xi_{1}+n,\xi_{2}) \frac{i^{n} t^{n}}{n!}.
\end{equation}

The cumulative distribution function, or probability distribution function, of our proposed extended beta distribution \eqref{4.1} can be expressed as
\begin{equation*}
F(x)=P[X<x]=\int_{0}^{x} f(x)\,dx,
\end{equation*}
so that
\begin{equation}\label{4.8}
F(x)=\frac{B_{\eta}^{p,q,x}(\xi_{1},\xi_{2})}{B_{\eta}^{p,q}(\xi_{1},\xi_{2})},
\end{equation}
where $B_{\eta}^{p,q,x}(\xi_{1},\xi_{2})$ denotes the (lower) incomplete extended beta function defined by
\begin{equation*}
B_{\eta}^{p,q,x}(\xi_{1},\xi_{2})=\int_{0}^{x} y^{\xi_{1}-1}~(1-y)^{\xi_{2}-1}~S_{\eta}\left[-\frac{p}{y}\right]S_{\eta}\left[-\frac{q}{1-y}\right]dy
\end{equation*}
\begin{equation*}
  (p,q > 0, -\infty <\xi_{1},\xi_{2} < \infty ~~,\Re(\eta)>-1).
\end{equation*}

The reliability function (which is simply the complement of the cumulative distribution function) of our proposed distribution is given by
\begin{equation*}
R(x)=P[X\geq x]=1-F(x)=\int_{x}^{\infty} f(x) dx
\end{equation*}
so that
\begin{equation}\label{4.9}
R(x)=\frac{{\hat B}_{\eta}^{p,q,x}(\xi_{1},\xi_{2})}{B_{\eta}^{p,q}(\xi_{1},\xi_{2})},
\end{equation}
where ${\hat B}_{\eta}^{p,q,x}(\xi_{1},\xi_{2})$ is the (upper) incomplete extended beta function defined by
\begin{equation*}
B_{\eta}^{p,q,x}(\xi_{1},\xi_{2})=\int_{x}^{\infty} y^{\xi_{1}-1}~(1-y)^{\xi_{2}-1}~S_{\eta}\left[-\frac{p}{y}\right]S_{\eta}\left[-\frac{q}{1-y}\right]dy
\end{equation*}
\begin{equation*}
  (p,q > 0, -\infty <\xi_{1},\xi_{2} < \infty ~~,\Re(\eta)>-1).
\end{equation*}

%%%
%%% Section 5
%%%
\section{\bf Extended hypergeometric functions and their associated properties}
In this section, we present the following extensions of the Gauss and confluent hypergeometric functions by making use of our extended beta function $B_{\eta}^{p,q}(\xi_{1},\xi_{2})$:
\begin{definition}
A new extension of the Gauss hypergeometric function is defined as follows:
\begin{equation} \label{5.1}
 \aligned & F_{\eta}^{p,q}(\xi_{1}, \xi_{2}; \xi_{3}; x)=\sum_{l=0}^{\infty} \frac{(\xi_{1})_{l}~B_{\eta}^{p,q}(\xi_{2}+l, \xi_{3}-\xi_{2})}{B(\xi_{2}, \xi_{3}-\xi_{2})} \frac{x^{l}}{l!}\\
& \hskip 7mm (p,q\geq0,~|x|<1,~\Re(\xi_{3})>\Re(\xi_{2})>0,~\Re(\eta)>-1).
 \endaligned
\end{equation}
\end{definition}
\begin{definition}
A new extension of the confluent hypergeometric function is defined as follows:
\begin{equation} \label{5.2}
\aligned & \Phi_{\eta}^{p,q}(\xi_{2}; \xi_{3}; x)=\sum_{l=0}^{\infty} \frac{B_{\eta}^{p,q}(\xi_{2}+l, \xi_{3}-\xi_{2})}{B(\xi_{2},
\xi_{3}-\xi_{2})}~\frac{x^{l}}{l!}\\
& \hskip 7mm  (p,q \geq0,~|x|<1,~\Re(\xi_{3})>\Re(\xi_{2})>0,~\Re(\eta)>-1).
\endaligned
\end{equation}
\end{definition}

\begin{remark}
We note that the case $\eta=-\frac{1}{2}$ in \eqref{5.1} and \eqref{5.2} yields the known extended Gauss and confluent hypergeometric functions defined by Choi et al. \cite{3}, which further for $q=p$ gives the known extension of the Gauss and confluent hypergeometric functions given by Chaudhry et al. \cite{2}. Clearly, for $p=q=0$, \eqref{5.1} and \eqref{5.2} reduce to the classical Gauss and confluent hypergeometric functions \cite{8}.
\end{remark}

\begin{theorem}
The following integral representations for our extended Gauss and confluent hypergeometric functions hold true:
\begin{equation} \label{5.3}
 \aligned &
 F_{\eta}^{p,q}(\xi_{1}, \xi_{2}; \xi_{3}; x)=\frac{1}{B(\xi_{2}, \xi_{3}-\xi_{2})}\\
 & \hskip 3mm \times\int_{0}^{1}
y^{\xi_{2}-1}~(1-y)^{\xi_{3}-\xi_{2}-1}~(1-yx)^{-\xi_{1}}  S_{\eta}\left[-\frac{p}{y}\right] S_{\eta}\left[-\frac{q}{1-y}\right]dy \\
&\hskip 15mm  (p,q,\geq0,~|\arg(1-x)|<\pi,~\Re(\xi_{3})>\Re(\xi_{2})>0,~\Re(\eta)>-1)
 \endaligned
\end{equation}
and
\begin{equation} \label{5.4}
 \aligned &
\Phi_{\eta}^{p,q}\xi_{2}; \xi_{3}; x)=\frac{1}{B(\xi_{2}, \xi_{3}-\xi_{2})} \\ & \hskip 3mm \times \int_{0}^{1}
y^{\xi_{2}-1}~(1-y)^{\xi_{3}-\xi_{2}-1}~e^{xy}~S_{\eta}\left[-\frac{p}{y}\right]S_{\eta}\left[-\frac{q}{1-y}\right]dy \\
& \hskip 20mm (p,q\geq0, ~\Re(\xi_{3})>\Re(\xi_{2})>0,~\Re(\eta)>-1).
 \endaligned
\end{equation}
\end{theorem}

\begin{proof} Each of the above representations can be readily established by using the integral representation of the extended beta function in \eqref{2.1} on the right-hand sides of \eqref{5.1} and \eqref{5.2}, respectively.
\end{proof}
\begin{theorem} The following integral representation holds true:
\begin{equation} \label{5.5}
 \aligned &
\Phi_{\eta}^{p,q}(\xi_{2}; \xi_{3}; x)=\frac{\exp(x)}{B(\xi_{2}, \xi_{3}-\xi_{2})} \\ & \hskip 3mm \times \int_{0}^{1}
(1-y)^{\xi_{2}-1}~y^{\xi_{3}-\xi_{2}-1}~e^{-xy}~S_{\eta}\left[-\frac{p}{y}\right]S_{\eta}\left[-\frac{q}{1-y}\right]dy \\
& \hskip 20mm (p,q\geq0, ~\Re(\xi_{3})>\Re(\xi_{2})>0,~\Re(\eta)>-1).
 \endaligned
\end{equation}
\end{theorem}
\begin{proof}
On replacing $y$ by $1-y$ in \eqref{5.4}, we easily get our desired result \eqref{5.5}.
\end{proof}

\begin{theorem}
The following differential formulas for the extended Gauss and confluent hypergeometric functions hold true:
\begin{equation} \label{5.6}
\frac{d^{k}}{dx^{k}}\left\{F_{\eta}^{p,q}(\xi_{1},\xi_{2};\xi_{3};x)\right\}=\frac{(\xi_{1})_{k}(\xi_{2})_{k}}{(\xi_{3})_{k}}F_{\eta}^{p,q}(\xi_{1}+k,\xi_{2}+k;\xi_{3}+k;x)
\end{equation}
$$(p,q\geq 0,~\Re(\eta)>-1,~k\in {\mathbb{N}}_{0})$$
and
\begin{equation} \label{5.7}
\frac{d^{k}}{dx^{k}}\left\{\Phi_{\eta}^{p,q}(\xi_{2};\xi_{3};x)\right\}=\frac{(\xi_{2})_{k}}{(\xi_{3})_{k}}\Phi_{\eta}^{p,q}(\xi_{2}+k;\xi_{3}+k;x)
\end{equation}
$$(p,q\geq 0,~\Re(\eta)>-1,~k\in {\mathbb{N}}_{0}).$$
\end{theorem}

\begin{proof} On differentiating \eqref{5.1} with respect to $x$, we obtain
\begin{equation*}
 \frac{d}{dx}\left\{F_{\eta}^{p,q}(\xi_{1}, \xi_{2}; \xi_{3}; x)\right\}=\sum_{l=1}^{\infty} \frac{(\xi_{1})_{l}~B_{\eta}^{p,q}(\xi_{2}+l, \xi_{3}-\xi_{2})}{B(\xi_{2}, \xi_{3}-\xi_{2})} \frac{x^{l-1}}{(l-1)!}.
\end{equation*}
On replacing $l$ by $l+1$, we then have
\begin{equation*}
 \frac{d}{dx}\left\{F_{\eta}^{p,q}(\xi_{1}, \xi_{2}; \xi_{3}; x)\right\}=\sum_{l=0}^{\infty} \frac{(\xi_{1})_{l+1}~B_{\eta}^{p,q}(\xi_{2}+l+1, \xi_{3}-\xi_{2})}{B(\xi_{2}, \xi_{3}-\xi_{2})} \frac{x^{l}}{l!}.
\end{equation*}

Now by using $B(\xi_{2}, \xi_{3}-\xi_{2})=\frac{\xi_{3}}{\xi_{2}}B(\xi_{2}+1, \xi_{3}-\xi_{2})$ and $(\xi_{1})_{l+1}=\xi_{1}(\xi_{1}+1)_{l}$, on the right-hand side of the above equation, we find
\begin{equation}\label{5.7c}
 \frac{d}{dx}\left\{F_{\eta}^{p,q}(\xi_{1}, \xi_{2}; \xi_{3}; x)\right\}=\frac{\xi_{1} \xi_{2}}{\xi_{3}}\sum_{l=0}^{\infty} \frac{(\xi_{1}+1)_{l}~B_{\eta}^{p,q}(\xi_{2}+l+1, \xi_{3}-\xi_{2})}{B(\xi_{2}+1, \xi_{3}-\xi_{2})} \frac{x^{l}}{l!}
\end{equation}
\begin{equation*}
 =\frac{\xi_{1} \xi_{2}}{\xi_{3}}F_{\eta}^{p,q}(\xi_{1}+1,\xi_{2}+1;\xi_{3}+1;x).
\end{equation*}
Again differentiating \eqref{5.7c} with respect to $x$, we have
\begin{equation*}
 \frac{d^{2}}{dx^{2}}\left\{F_{\eta}^{p,q}(\xi_{1}, \xi_{2}; \xi_{3}; x)\right\}=\frac{\xi_{1}(\xi_{1}+1) \xi_{2}(\xi_{2}+1)}{\xi_{3}(\xi_{3}+1)}F_{\eta}^{p,q}(\xi_{1}+2,\xi_{2}+2;\xi_{3}+2;x).
\end{equation*}
Continuing this process, by induction we obtain the required result \eqref{5.6}.
Similarly, we can establish the result \eqref{5.7}.
\end{proof}

\begin{theorem}
The following transformation formulas for the extended Gauss and confluent hypergeometric functions hold true:
\begin{equation} \label{5.8}
F_{\eta}^{p,q}(\xi_{1},\xi_{2};\xi_{3};x)=(1-x)^{-\xi_{1}}F_{\eta}^{p,q}\left(\xi_{1},\xi_{3}-\xi_{2};\xi_{2}; -\frac{x}{(1-x)}\right)
\end{equation}
$$(p,q\geq 0,~\Re(\eta)>-1, |\arg\,(1-x)|<\pi)$$
and
\begin{equation} \label{5.9}
\Phi_{\eta}^{p,q}(\xi_{2};\xi_{3};x)=\exp(x)\Phi_{\eta}^{p,q}\left(\xi_{3}-\xi_{2};\xi_{3}; -x\right)
\end{equation}
$$(p,q\geq 0,~\Re(\eta)>-1).$$
\end{theorem}

\begin{proof}
On replacing $y$ by$1-y$ in \eqref{5.3} and then using $[1-x(1-y)]^{-\xi_{1}}=(1-x)^{-\xi_{1}}\left[1+\frac{x}{1-x}y\right]^{-\xi_{1}}$, we have

\begin{equation*}
F_{\eta}^{p,q}(\xi_{1}, \xi_{2}; \xi_{3}; x)=\frac{(1-x)^{-\xi_{1}}}{B(\xi_{2}, \xi_{3}-\xi_{2})}
\end{equation*}
\begin{equation*}
\times\int_{0}^{1}
y^{\xi_{3}-\xi_{2}-1}~(1-y)^{\xi_{2}-1}~\left(1+\frac{x}{1-x}y\right)^{-\xi_{1}}  S_{\eta}\left[-\frac{p}{y}\right]S_{\eta}\left[-\frac{q}{1-y}\right]dy,
\end{equation*}
which in view of \eqref{5.3}, yields the right-hand side of \eqref{5.8}. In a similar way, we can establish \eqref{5.9}.
\end{proof}

\begin{theorem}
The following generating function for the extended Gauss hypergeometric function holds true:
\begin{equation} \label{5.11}
\sum_{k=0}^{\infty} (\xi_{1})_{k}~F_{\eta}^{p,q}(\xi_{1}+k,\xi_{2};\xi_{3};x)\frac{z^{k}}{k!}=(1-z)^{-\xi_{1}} F_{\eta}^{p,q}\left(\xi_{1},\xi_{2};\xi_{3};\frac{x}{1-z}\right)
\end{equation}
$$(p,q\geq 0,~\Re(\eta)>-1,~|z|<1).$$
\end{theorem}

\begin{proof}
Let $\Im$ be the left-hand side of \eqref{5.11}. By the virtue of \eqref{5.1}, we have
\begin{equation*}
\Im=\sum_{k=0}^{\infty} (\xi_{1})_{k}\left[\sum_{l=0}^{\infty} \frac{(\xi_{1}+k)_{l}~B_{\eta}^{p,q}(\xi_{2}+l, \xi_{3}-\xi_{2})}{B(\xi_{2}, \xi_{3}-\xi_{2})} \frac{x^{l}}{l!}\right]\frac{z^{k}}{k!}.
\end{equation*}
Now by using the identity $(\xi_{1})_{k}(\xi_{1}+k)_{l}=(\xi_{1})_{l}(\xi_{1}+l)_{k}$ in the above expression, we obtain
\begin{equation*}
\Im=\sum_{l=0}^{\infty} (\xi_{1})_{l} \frac{B_{\eta}^{p,q}(\xi_{2}+l, \xi_{3}-\xi_{2})}{B(\xi_{2}, \xi_{3}-\xi_{2})}\left[\sum_{k=0}^{\infty} (\xi_{1}+l)_{k} \frac{z^{k}}{k!}\right]\frac{x^{l}}{l!}.
\end{equation*}
On applying the binomial theorem to the inner summation, we obtain
\begin{equation*}
\Im=\sum_{l=0}^{\infty} (\xi_{1})_{l} \frac{B_{\eta}^{p,q}(\xi_{2}+l, \xi_{3}-\xi_{2})}{B(\xi_{2}, \xi_{3}-\xi_{2})}(1-z)^{-(\xi_{1}+l)} \frac{x^{l}}{l!},
\end{equation*}
which upon further use of \eqref{5.1} yields the stated result \eqref{5.11}.

\end{proof}


\begin{thebibliography}{20}

\bibitem{m1} A. A. Al-Gonah and W. K. Mohammed,
\newblock A new extension of extended gamma and beta functions and their properties,
\newblock {\em J. Scint. Engg. Rese.,}
\textbf{9} (2018), 257--270.

\bibitem{m2} M. Ali and M. Ghayasuddin,
\newblock A note on extended beta, Gauss and confluent hypergeometric functions,
\newblock {\em Italian J. Pure and Appl. Math.,}
2020. (Accepted)

\bibitem{3} J. Choi, A. K. Rathie, R. K. Parmar,
\newblock Extension of extended beta, hypergeometric and
confluent hypergeometric functions,
\newblock {\em Honam Mathematical J.}
\textbf{36}(2) (2014), 357--385.

\bibitem{1}
M. A. Chaudhry, A. Qadir, M. Rafique and S. M. Zubair,
\newblock Extension of Euler's beta function,
\newblock {\em J. Comput. Appl. Math.}
\textbf{78}(1) (1997), 19--32.

\bibitem{2}
M. A. Chaudhry, A. Qadir, H. M. Srivastava and R. B. Paris,
\newblock Extended hypergeometric and confluent hypergeometric functions,
\newblock {\em Appl. Math. Comput.}
\textbf{159}(2) (2004), 589--602.

\bibitem{4}
A. Gasmi and M. Sifi,
\newblock The Bessel-Struve interwining operator on C and mean periodic functions,
\newblock {\em IJMMS}
\textbf{59} (2004), 3171--3185.

\bibitem{m3} M. Ghayasuddin, N. U. Khan and M. Ali,
\newblock A study on extended beta, Gauss and confluent hypergeometric functions,
\newblock {\em Intern. J. Appl. Math.,}
2020.

\bibitem{5}
N. U. Khan, S. W. Khan and M. Ghayasuddin,
\newblock Some new results associated with the Bessel-Struve kernel function,
\newblock {\em Acta Uni. Apul.}
\textbf{48} (2016), 89--101.

\bibitem{nta}
N. U. Khan, T. Usman and M. Aman,
\newblock Extended beta, hypergeometric and confluent hypergeometric functions,
\newblock {\em Trans. Natl. Acad. Sci. Azerb. Ser. Phys.-Tech. Math. Sci.}
\textbf{39}(1) (2019), 1--16.

\bibitem{6}
E. $\ddot{\rm O}$zergin, M. A. $\ddot{\rm O}$zarslan and A. Altin,
\newblock Extension of gamma, beta and hypergeometric functions,
\newblock {\em J. Comput. Appl. Math.}
\textbf{235} (2011), 4601--4610.

\bibitem{7} R. K. Parmar,
\newblock A new generalization of Gamma, Beta, hypergeometric and confluent hypergeometric functions,
\newblock {\em Le Matematiche}
\textbf{LXVIII} (2013), 33--52.

\bibitem{8}
 E. D. Rainville,
\newblock Special functions,
\newblock {\em Macmillan Company, New York,} 1960.
\newblock {\em Reprinted by Chelsea Publishing Company, Bronx, New York,} 1971.


\bibitem{10}
H. M. Srivastava and H. L. Manocha,
\newblock A treatise on generating functions,
\newblock {\em Halsted Press (Ellis Horwood Limited, Chichester), John Wiley and Sons, New
York, Chichester, Brisbane and Toronto}
1984.


\bibitem{9} M. Shadab, S. Jabee and J. Choi,
\newblock An extension of beta function and its application,
\newblock {\em Far East Journal of Mathematical Sciences}
\textbf{103}(1) (2018), 235--251.







\end{thebibliography}
\end{document}